\documentclass[12pt]{amsart}
\usepackage{amsfonts,latexsym,amssymb,
}
\newtheorem{theorem}{Theorem}[section]
\newtheorem{lemma}[theorem]{Lemma}

\newtheorem{proposition}[theorem]{Proposition} 
 
\newtheorem{corollary}[theorem]{Corollary} 
\newtheorem{fact}[theorem]{Fact} 

\theoremstyle{definition}
\newtheorem{definition}[theorem]{Definition}

\theoremstyle{remark}
\newtheorem{remark}[theorem]{Remark}

\newcommand{\brfr}{$\hspace{0 pt}$}
\newcommand{\brfrt}{\hspace{0 pt}}
\DeclareMathOperator{\cf}{cf}

\DeclareMathOperator{\rk}{RK}

\newcommand{\m}{\mathfrak}

\newcommand{\bigbox}{{ \Box}} % 

\title [generalizations of pseudocompactness]
{More generalizations of pseudocompactness}

\author[]{Paolo Lipparini} 
\address{Dipartimento  Matematic\ae\\
Viale della Ricerca Scientifica\\
II Universit\`a di Roma (Tor Vergata)\\
I-00133 ROME 
ITALY}
\urladdr{http://www.mat.uniroma2.it/\textasciitilde lipparin}

\thanks{The author has received support from MPI and GNSAGA.
We wish to express our gratitude to X. Caicedo
and S. Garcia-Ferreira for stimulating discussions and correspondence} 



\begin{document} 

\begin{abstract} 
We introduce  a covering notion depending on two cardinals, 
which we call $\mathcal O  $-$ [ \mu, \lambda ]$-compactness,
and which
encompasses both pseudocompactness and many other generalizations 
of pseudocompactness.
For Tychonoff spaces, pseudocompactness turns out to be equivalent
to $\mathcal O  $-$ [ \omega, \omega  ]$-compactness.

We provide several characterizations of $\mathcal O  $-$ [ \mu, \lambda ]$-compactness, and we discuss 
its connection with $D$-pseudocompactness, for $D$ an ultrafilter. We analyze the
behaviour of the above notions with respect to products.

Finally, we show that our results hold in a more general framework,
in which compactness properties are defined relative to an arbitrary family 
 of subsets of some topological space $X$.
\end{abstract}

\maketitle  

\section{Introduction} \label{intro} 

As well-known, there are many equivalent reformulations of pseudocompactness. See, e. g. \cite{St}. Various generalizations and extensions of pseudocompactness have been introduced by many authors; see, among others, \cite{Ar, CNcc,Fr,GF,GS,Gl,Ke,LiF,Re,Sa,SS2,SS,St,SV}.
We introduce here some more pseudocompactness-like properties, focusing mainly
on  notions related to covering properties and ultrafilter convergence. 

The most  general form of our notion 
depends on
two cardinals $\mu$ and $\lambda$;
we call it
$\mathcal O$-$ [ \mu, \lambda ]$-compactness.
It generalizes and unifies several
pseudocompactness-like notions appeared before.
See Remark \ref{precedenti}.
In a sense, $\mathcal O$-$ [ \mu, \lambda ]$-compactness
is to pseudocompactness what $ [ \mu, \lambda ]$-compactness
is to countable compactness. See Remark \ref{mlcp}.
In particular, for Tychonoff spaces,
$\mathcal O  $-$ [ \omega, \omega  ]$-compactness
turns out to be equivalent to 
pseudocompactness.

We find many conditions equivalent to $\mathcal O  $-$ [ \mu, \lambda ]$-compactness. In particular, a characterization by means
of ultrafilters, Theorem \ref{ufo},  plays an important role in this paper.
It provides a connection between 
$\mathcal O  $-$ [ \mu, \lambda ]$-compactness
and $D$-pseudocompactness, for $D$  a $(\mu, \lambda )$-regular ultrafilter.
The notion of a $(\mu, \lambda )$-regular ultrafilter
arose in a model-theoretical setting, and has proved useful
also in some areas of set-theory, and even in topology.
See \cite{mru,topappl} for references.

More sophisticated results are involved when we deal with products,
since $D$-pseudocompactness is productive, but 
$\mathcal O  $-$ [ \mu, \lambda ]$-compactness is not productive,
as well known in the special case $ \mu= \lambda = \omega $, that is, pseudocompactness.
We show that if $D$ is a 
$(\mu, \lambda )$-regular ultrafilter, then
every  $D$-pseudocompact topological space $X$ is
$\mathcal O  $-$ [ \mu, \lambda ]$-compact, hence all
(Tychonoff) powers of $X$ are 
$\mathcal O  $-$ [ \mu, \lambda ]$-compact, too (Corollary \ref{regimpcpn}).
The situation is in part parallel to the relationship between
the more classical notions of $D$-compactness and $ [ \mu, \lambda ]$-compactness.
In this latter case, an equivalence holds: all powers 
of a topological space $X$ are $ [ \mu, \lambda ]$-compact 
if and only if there is some
$(\mu, \lambda )$-regular ultrafilter $D$ such that 
$X$ is  $D$-compact.
We show that an analogous result holds for $D$-pseudocompactness,
provided we deal with a notion slightly stronger than 
$\mathcal O  $-$ [ \mu, \lambda ]$-compactness. See 
Definition \ref{fcpnopbox} and Theorem \ref{fprodo}. 
In particular, we provide a characterization of those spaces  
which are 
$D$-pseudocompact, for
some
$(\mu, \lambda )$-regular ultrafilter $D$.

In the final section of this note we mention that our results generalize to the abstract framework presented in \cite{LiF}. That is, our proofs work essentially unchanged both for pseudocompactness-like notions and for the corresponding compactness notions.
 In \cite{LiF} each compactness property is defined relative to a family 
$\mathcal F$ of subsets of some topological space $X$. 
The pseudocompactness case is obtained when
$\mathcal F=\mathcal O$, the family of all nonempty open sets of $X$. 
When $\mathcal F$ is the family of all singletons of $X$, we obtain 
results related to $[ \mu, \lambda ]$-compactness. 

\medskip

Our notation is fairly standard.
Unless
explicitly mentioned, we assume no separation axiom.
However, the reader is warned that there are many conditions equivalent to
pseudocompactness, but the equivalence holds only assuming some separation 
axiom (they are all equivalent only for Tychonoff spaces).
For Tychonoff spaces, the particular case $ \mu= \lambda = \omega $
of our definitions of 
$\mathcal O$-$ [ \mu, \lambda ]$-compactness
(Definition \ref{fcpnop}) turns out to be equivalent to pseudocompactness, 
but this is not necessarily the case for spaces with lower separation properties.
See Remark \ref{precedenti}.

\section{A two cardinal generalization of pseudocompactness} \label{genps}

The following definition originally appeared in \cite{LiF}
in a more general framework. The letter $\mathcal O$  is intended to 
denote the family of all the nonempty open sets of some topological space $X$. In this sense, the definition of $\mathcal O$-$ [ \mu, \lambda ]$-compactness
is the particular case $\mathcal F= \mathcal O$  of the definition of 
$\mathcal F$-$ [ \mu, \lambda ]$-compactness in \cite[Definition 4.2]{LiF}.
See also Section \ref{abstrfr}.

\begin{definition} \label{fcpnop}
We say that a topological space $X$ is
$\mathcal O$-$ [ \mu, \lambda ]$-\emph{compact}
if and only if the following holds.

For every sequence $( C _ \alpha ) _{ \alpha \in \lambda } $
 of closed sets of $X$, if,
for every $Z \subseteq \lambda $ with $ |Z|< \mu$,
there exists a nonempty open set $O_Z$ of $X$   such that   
$ \bigcap _{ \alpha \in Z}  C_ \alpha \supseteq O_Z$,
then  
$ \bigcap _{ \alpha \in \lambda }  C_ \alpha \not= \emptyset $.

Clearly, in the above definition, we can equivalently let
$O_Z$ vary among the (nonempty) elements of some base of $X$, rather
than among all  nonempty open sets.
Also, by considering complements, we have that 
$\mathcal O$-$ [ \mu, \lambda ]$-compactness
is equivalent to the following statement.

For every $\lambda$-indexed open cover
 $( Q _ \alpha ) _{ \alpha \in \lambda } $
 of $X$, there exists $Z \subseteq \lambda $, with $ |Z|< \mu $,
such that
$ \bigcup _{ \alpha \in Z}  Q_ \alpha$
is dense in $X$.
\end{definition}

\begin{remark} \label{mlcp} 
The notion of 
$\mathcal O$-$ [ \mu, \lambda ]$-compactness
should be compared with the more classical notion of
$ [ \mu, \lambda ]$-compactness.

A topological space $X$ is
$ [ \mu, \lambda ]$-\emph{compact}
if and only if, 
for every sequence $( C _ \alpha ) _{ \alpha \in \lambda } $
 of closed sets of $X$, if
$ \bigcap _{ \alpha \in Z}  C_ \alpha \not= \emptyset  $,
for every $Z \subseteq \lambda $ with $ |Z|< \mu$,
then  
$ \bigcap _{ \alpha \in \lambda }  C_ \alpha \not= \emptyset $.

Thus, in the definition of $ [ \mu, \lambda ]$-compactness,
we require only the weaker assumption that
$ \bigcap _{ \alpha \in Z}  C_ \alpha  $ is nonempty, 
for every $Z \subseteq \lambda $ with $ |Z|< \mu$,
rather than requiring that $ \bigcap _{ \alpha \in Z}  C_ \alpha  $ contains some nonempty open set.
In particular, every 
$ [ \mu, \lambda ]$-compact space is
$\mathcal O$-$ [ \mu, \lambda ]$-compact.

Thus, $ [ \omega , \omega  ]$-compactness is the same as countable compactness,
which is the analogue of pseudocompactness for  $\mathcal O$-$ [ \mu, \lambda ]$-compactness.
Many of the results presented here are versions for $\mathcal O$-$ [ \mu, \lambda ]$-compactness
of known results about 
$ [ \mu, \lambda ]$-compactness. Indeed,
a simultaneous method of proof is available 
for both cases, and shall be mentioned in Section \ref{abstrfr}.

Notice that $ [ \mu, \lambda ]$-compactness is a notion which encompasses both 
Lindel\"ofness (more generally, $\kappa$-final compactness) and countable compactness (more generally, $\kappa$-initial compactness). See, e. g.,
\cite{C,G,topproc,topappl,Vfund} and references there for further information about 
$ [ \mu, \lambda ]$-compactness.
\end{remark}

\begin{remark} \label{precedenti}    
For Tychonoff spaces, 
$\mathcal O$-$ [ \omega , \omega ]$-compactness is equivalent to pseudocompactness. 
Without assuming $X$ to be Tychonoff, 
 $\mathcal O$-$ [ \omega , \omega ]$-compactness
turns out to be equivalent to a condition which is usually
called \emph{feeble compactness}. See \cite[Theorem 4.4(1) and Remark 4.5]{LiF} and \cite{St}.
 
More generally, the particular case $ \mu= \omega $ of Definition 
\ref{fcpnop},
that is,  
$\mathcal O$-$ [ \omega , \lambda  ]$-compactness,  
has been introduced and studied in 
\cite{Fr}, 
where it 
 is called
\emph{almost $\lambda$-compactness}.
The notion of $\mathcal O$-$ [ \omega , \lambda  ]$-compactness has also been
studied, under different names, in
\cite{SS2},  
as
  \emph{weak-$ \lambda $-$ \aleph_0$-compactness},
and in 
\cite{Re,SV} 
as
  \emph{weak initial $\lambda$-compactness}.

Moreover, \cite{Fr} introduced also a notion which corresponds to 
$\mathcal O$-$ [ \mu ,\lambda]$-compactness for all cardinals 
$\lambda$, calling it \emph{almost $\mu$-Lindel\"ofness}.

Assuming that
$X$ is a Tychonoff space,
  a property equivalent to 
$\mathcal O$-$ [ \kappa , \kappa  ]$-compactness,
has been introduced in \cite{CNcc} under the name \emph{pseudo-$( \kappa, \kappa )$-compactness}.  
%yyy controllare! 
See \cite[Theorem 4.4]{LiF}.

Definition \ref{fcpnop} generalizes all the above mentioned 
notions.

See \cite{Ar,CNcc,Fr,GF,GS,Gl,Ke,LiF,Re,Sa,SS2,SS,St,SV} for the study of further related notions.
\end{remark}

For $\lambda$, $\mu$ infinite cardinals, $S_ \mu ( \lambda )$
denotes the set of all subsets of $\lambda$ of cardinality $<\mu$.  
We put $ \lambda ^{<\mu} = \sup _{\mu' < \mu} \lambda ^{ \mu'}$.
Thus, $ \lambda ^{<\mu}$ is the cardinality of  $S_ \mu ( \lambda )$.

In the next proposition we present some useful
conditions equivalent to $\mathcal O$-$ [ \mu, \lambda ]$-compactness.
A further important characterization will be presented in Theorem \ref{ufo}.

\begin{proposition} \label{equivo} %2.4?
For every topological space $X$ and infinite cardinals $\lambda$ and $\mu$, 
the following are equivalent.

\begin{enumerate} 
\item
$X$ is
$\mathcal O$-$ [ \mu, \lambda ]$-compact.

\item
For every  sequence  $( P _ \alpha ) _{ \alpha \in \lambda } $
of subsets of $X$, if,
for every $Z \subseteq \lambda $ with $ |Z|< \mu$,
there exists a nonempty open set $O_Z$ of $X$   such that   
$ \bigcap _{ \alpha \in Z}  P_ \alpha \supseteq O_Z$,
then  
$ \bigcap _{ \alpha \in \lambda }  \overline{P}_ \alpha \not= \emptyset $.

\item
For every  sequence  $( Q _ \alpha ) _{ \alpha \in \lambda } $
 of open sets of $X$, if,
for every $Z \subseteq \lambda $ with $ |Z|< \mu$,
there exists a nonempty open set $O_Z$ of $X$   such that   
$ \bigcap _{ \alpha \in Z}  Q_ \alpha \supseteq O_Z$,
then  
$ \bigcap _{ \alpha \in \lambda }  \overline{Q}_ \alpha \not= \emptyset $.

\item
For every  sequence 
$\{ O_Z \mid Z \in S_ \mu ( \lambda ) \}$
of nonempty open sets of $X$,
it happens that
    $ \bigcap _{ \alpha \in \lambda } 
 \overline{  \bigcup  \{ O_Z \mid  Z \in S_ \mu ( \lambda ),  \alpha \in Z \}}
  \not= \emptyset $.

\item
For every  sequence 
$\{ O_Z \mid Z \in S_ \mu ( \lambda ) \}$
of nonempty open sets of $X$,
the following holds.
If, for every finite subset $W$  of $\lambda$,
we put  
$Q_W= 
  \bigcup \{ O_Z \mid Z \in S_ \mu ( \lambda ) \text{ and } Z \supseteq
W\}$,
then 
$\bigcap \{ \overline{Q_W} \mid W \text{ is a finite subset of } \lambda  \}\not= \emptyset $.

\item
For every  sequence 
$\{ C_Z \mid Z \in S_ \mu ( \lambda ) \}$
of closed sets of $X$, such that each
$C_Z$ is properly contained in $X$,
if we let, for $\alpha \in \lambda $,
$P_ \alpha $ be the interior
of 
$ \bigcap  \{ C_Z \mid  Z \in S_ \mu ( \lambda ),  \alpha \in Z \}$, then
we have that 
$(P_ \alpha ) _{ \alpha \in \lambda } $
is not a cover of $X$.  
\end{enumerate}
 \end{proposition}

\begin{proof}
(1) $\Rightarrow $  (2) 
Just take $C_ \alpha =\overline{P}_ \alpha$, for
$ \alpha \in \lambda $.

(2) $\Rightarrow $  (3) is trivial.

(3) $\Rightarrow $  (5)
The  sequence  $\{ Q_W \mid W \text{ is a finite subset of } \lambda  \}$
is a  sequence  of $\lambda$ open sets of $X$, since there are
$\lambda$ finite subsets of $\lambda$.

For every  $ \nu < \mu$,
if
$(W_ \beta ) _{ \beta \in \nu} $ is a  sequence  of finite subsets of $\lambda$,
then $Z = \bigcup _{ \beta \in \nu} W _ \beta $ has cardinality $ \leq \nu$,
and thus belongs to $S_ \mu ( \lambda ) $. Moreover, 
for each $ \beta \in \nu$, we have that $Z \supseteq W_ \beta $,
hence $Q _{W_ \beta } \supseteq O_Z $.  
This implies that
 $  \bigcap _{ \beta \in \nu} Q _{W_ \beta } \supseteq O_Z $.

We have proved that the  sequence  
$\{ Q_W \mid W \text{ a finite subset of } \lambda  \}$
is a  sequence  of $\lambda$ open sets of $X$ such that 
the intersection of $<\mu$ members of the  sequence  
contains some nonempty open set of $X$.
By applying (3) to this sequence, we have that
$\bigcap \{ \overline{Q_W} \mid W \text{ is a finite subset of } \lambda  \}\not= \emptyset $.

(5) $\Rightarrow $  (4) is trivial.

(4) $\Rightarrow $  (1)
Suppose that 
$( C _ \alpha ) _{ \alpha \in \lambda } $
and 
$O_Z$,
for  $Z \subseteq \lambda $ with $ |Z|< \mu$,
are as in the premise of the definition of 
$\mathcal O$-$ [ \mu, \lambda ]$-compactness.

For $ \alpha \in \lambda $, let
$C'_ \alpha = 
\overline{  \bigcup  \{ O_Z \mid  Z \in S_ \mu ( \lambda ),  \alpha \in Z \}}$. Since  
$C_ \alpha $ is closed, and 
$C_ \alpha \supseteq O_Z$
whenever  $ \alpha \in Z$, we have that
$C_ \alpha \supseteq C'_ \alpha  $.
By (4), 
$ \bigcap _{ \alpha \in \lambda }  C'_ \alpha \not= \emptyset $, hence also
$ \bigcap _{ \alpha \in \lambda }  C_ \alpha \not= \emptyset $.
Thus we have proved (1). 

We shall also give a direct proof of (3) $\Rightarrow $  (4),
since it is very simple.
Given the  sequence 
$\{ O_Z \mid  Z \in S_ \mu ( \lambda ) \} $
then, for every $ \alpha \in \lambda $,
put   $Q_ \alpha =  \bigcup  \{ O_Z \mid  Z \in S_ \mu ( \lambda ),  \alpha \in Z \} $.
 For every $Z \in S_ \mu ( \lambda )$,
and every   $ \alpha \in Z$,
we have that  $Q_ \alpha \supseteq  O_Z $.
Hence, 
for every $Z \in S_ \mu ( \lambda )$,
we get $ \bigcap _{ \alpha \in Z}  Q_ \alpha \supseteq O_Z$,
so that we can apply (3).

(4) $\Leftrightarrow $  (6)
is immediate by taking complements.
 \end{proof}

In the particular case when
$\mu=\lambda$ is regular,
there are many more conditions equivalent 
to 
$\mathcal O$-$ [ \lambda , \lambda ]$-compactness.

\begin{theorem} \label{equivcpnrego}
Suppose that $X$ is a topological space,
and $ \lambda $ is a regular cardinal. Then
the following conditions are equivalent.
 
(a) $X$ is $\mathcal O$-$ [ \lambda , \lambda ]$-compact.

(b) Suppose that  $( C _ \alpha ) _{ \alpha \in \lambda } $
is a  sequence  of closed sets of $X$ such that $C_ \alpha \supseteq C_ \beta $,
whenever $ \alpha \leq \beta  < \lambda $.
If, for every $ \alpha \in \lambda $, there exists   
a nonempt open set $O$ of $X$ such that   
$  C_ \alpha \supseteq O$,
then  
$ \bigcap _{ \alpha \in \lambda }  C_ \alpha \not= \emptyset $.

(c) Suppose that  $( C _ \alpha ) _{ \alpha \in \lambda } $
is a  sequence  of closed sets of $X$ such that $C_ \alpha \supseteq C_ \beta $,
whenever $ \alpha \leq \beta  < \lambda $.
Suppose further that, for every $ \alpha \in \lambda $,
$C _ \alpha $ is the closure of some open set of $X$.   
If, for every $ \alpha \in \lambda $, there exists   
a nonempt open set $O$ of $X$ such that   
$  C_ \alpha \supseteq O$,
then  
$ \bigcap _{ \alpha \in \lambda }  C_ \alpha \not= \emptyset $.

(d) For every sequence $ (O _ \alpha ) _{ \alpha \in \lambda } $  
of nonempty open sets of $X$,
there exists $x \in X$ such that 
$|\{ \alpha \in \lambda  \mid U \cap O_ \alpha  \not= \emptyset  \}|= \lambda $,
for every neighborhood $U$ of $x$ in $X$. 

(e) For every sequence $ (O _ \alpha ) _{ \alpha \in \lambda } $  of 
nonempty open sets of $X$,
 there exists some ultrafilter $D$ uniform over $\lambda$ such that
 $ (O _ \alpha ) _{ \alpha \in \lambda } $ has a $D$-limit point
(see Definition \ref{dlim}).

(f) For every $\lambda$-indexed open cover
 $( O _ \alpha ) _{ \alpha \in \lambda } $
 of $X$, such that $O_ \alpha \subseteq O_ \beta $
whenever $ \alpha \leq \beta  < \lambda $, 
there exists $ \alpha \in \lambda$ such that 
$O_ \alpha $ is dense in $X$.

In all the above statements we can equivalently require that
the elements of the sequence  $( C _ \alpha ) _{ \alpha \in \lambda } $, respectively,
 $( O _ \alpha ) _{ \alpha \in \lambda } $, are all distinct.
\end{theorem}

\begin{proof}
By \cite[Theorem 4.4]{LiF},
taking $\mathcal F$ there to be the family $\mathcal O$  of all the 
nonempty open sets of $X$.

Since $\lambda$ is regular, the last statement is trivial,
as far as conditions (b), (c) and (f) are concerned. It follows from
\cite[Proposition 3.3(a)]{LiF} in case (d).
Then apply \cite[Proposition 4.1 ]{LiF} in order to get
(e).  
\end{proof}

\begin{remark} \label{notk}
At this point, we should mention a significant difference
between 
$\mathcal O$-$ [ \mu , \lambda ]$-compactness and
$ [ \mu, \lambda ]$-compactness.

It is true that a topological
space  is $ [ \mu, \lambda ]$-compact if and only if
it is $ [ \kappa , \kappa  ]$-compact, for every $\kappa$ such that 
$\mu \leq \kappa \leq \lambda $. 
Though simple, the above equivalence has proved very useful
in many circumstances. See, e. g., \cite{topappl}. 

It is trivial that every $\mathcal O$-$ [ \mu , \lambda ]$-compact space 
is $\mathcal O$-$ [ \mu' , \lambda' ]$-compact, whenever
$\mu \leq \mu' \leq \lambda' \leq \lambda $. In particular, 
every $\mathcal O$-$ [ \mu , \lambda ]$-compact space
is $\mathcal O$-$ [ \kappa , \kappa  ]$-compact, for every $\kappa$ such that 
$\mu \leq \kappa \leq \lambda $.

On the contrary, the condition of being
 $\mathcal O$-$ [ \kappa , \kappa  ]$-compact, for every $\kappa$ such that 
$\mu \leq \kappa \leq \lambda $, 
is not always a sufficient condition in order to get
$\mathcal O$-$ [ \mu , \lambda ]$-compactness.
 See Remark \ref{gf}. This fact limits
the usefulness of Theorem \ref{equivcpnrego} in the present context.
\end{remark}

\section{A characterization by means of ultrafilters} \label{uf} 

The first theorem in this section, Theorem \ref{ufo}, furnishes a characterization of $\mathcal O$-$ [ \mu, \lambda ]$-compactness by means of the existence of $D$-limit points of ultrafilters.
This characterization is the key for the study of the connections between $\mathcal O$-$ [ \mu, \lambda ]$-compactness and $D$-\brfrt pseudocompactness, for $D$ a
$( \mu, \lambda  )$-regular ultrafilter and shall be used in the next section in connection with properties of products.

\begin{definition} \label{dlim}
 Suppose that  $D$ is an ultrafilter over some set $I$, and
$X$ is a topological space.
If
$(Y_i)_{i \in I}$ is a  sequence  
of subsets of $X$, then $x \in X$
 is called a $D$-\emph{limit point} of $(Y_i)_{i \in I}$
if and only if 
$ \{ i \in I \mid  Y_i \cap U \not= \emptyset \} \in D$, for every neighborhood
$U$ of $x$ in $X$.
The  notion of a $D$-limit point is due to \cite[Definition 4.1]{GS}
for non-principal ultrafilters over $ \omega $,
and appears in \cite{GF} for uniform ultrafilters over arbitrary cardinals. 
 \end{definition}   

We say that an ultrafilter $D$ over  $S_ \mu ( \lambda )$
\emph{covers} $\lambda$ if and only if,
for every $\alpha \in \lambda $,
it happens that
$\{ Z \in S_ \mu ( \lambda ) \mid \alpha \in Z \} \in D$.
This notion is connected with $(\mu, \lambda )$-regularity,
as we shall see in Definition \ref{rk}. 

\begin{theorem} \label{ufo}
For every topological space $X$ and infinite cardinals $\lambda$ and $\mu$, 
the following are equivalent.

\begin{enumerate} 
\item
$X$ is
$\mathcal O$-$ [ \mu, \lambda ]$-compact.

\item
For every  sequence 
$\{ O_Z \mid Z \in S_ \mu ( \lambda ) \}$
of nonempty open sets of $X$,
there exists an ultrafilter $D$ over 
$S _\mu ( \lambda )$ which covers $\lambda$ and 
such that 
 $\{ O_Z \mid Z \in S_ \mu ( \lambda ) \}$
has a $D$-limit point.
\end{enumerate}
 \end{theorem}

\begin{proof} 
(1) $\Rightarrow $  (2)
Suppose that $\{ O_Z \mid Z \in S_ \mu ( \lambda ) \}$
is a  sequence  of nonempty open sets of $X$.
For every finite subset $W$ of $\lambda$,
let 
$Q_W= 
  \bigcup \{ O_Z \mid Z \in S_ \mu ( \lambda ) \text{ and } Z \supseteq
W\}$. 
By 
$\mathcal O$-$ [ \mu, \lambda ]$-compactness, and
Condition (5)
in Proposition \ref{equivo}, we have that
$\bigcap \{ \overline{Q_W} \mid W \text{ a finite subset of } \lambda  \}
\brfr
\not= \emptyset $.
Suppose that 
$x \in \bigcap \{ \overline{Q_W} \mid W \text{ a finite subset of } \lambda  \}$.

For every neighborhood $U$ of $x$ in $X$, let 
$A_U= \{ Z \in S_ \mu ( \lambda ) \mid U \cap O_Z \not= \emptyset\}  $. 
For every $\alpha \in \lambda $, let
 $[ \alpha ) =\{ Z \in S_ \mu ( \lambda ) \mid \alpha \in Z \}$.
We are going to show that the family
$\mathcal A= \{ [ \alpha )  \mid  \alpha \in \lambda \}
\cup \{ A_U \mid $U$ \text{ a neighborhood of } x \text{ in } X \}  $
has the finite intersection property.

Indeed, let $U_1, \dots, U_n$ be neighborhoods of $x$,
and $\alpha_1, \dots, \alpha _m$ be elements of $\lambda$.
Let $U= U_1 \cap \dots \cap U_n$,
$W= \{ \alpha_1, \dots, \alpha _m\} $,
and   $[W) =  [\alpha_1) \cap \dots \cap [\alpha _m) = 
\{ Z \in S_ \mu ( \lambda ) \mid  Z \supseteq W\}$.
Since $x \in  \overline{Q_W}$, we get that
$U \cap Q_W \not= \emptyset $, that is,
 $U \cap O_Z \not= \emptyset $, for some 
$Z \in S_ \mu ( \lambda )$ with 
$Z \supseteq W$. 
Hence $Z \in A_U$, and also
  $Z \in A _{U_1} $, \dots, 
$Z \in A _{U_n}$, since
$U_1 \supseteq U$, \dots,  
$U_n \supseteq U$.
In conclusion, 
$Z \in A _{U_1} \cap \dots \cap A _{U_n} \cap  [\alpha_1) \cap \dots \cap [\alpha _m)  $, hence the above intersection is not empty.

We have showed that  
$\mathcal A$
has the finite intersection property.
Hence $\mathcal A$ can be extended to some ultrafilter 
$D$ over $S_ \mu ( \lambda )$. 
By construction, 
$[ \alpha )   \in D$,
for every $\alpha \in \lambda $,
hence $D$ covers $\lambda$.
Again by construction, 
$A_U \in D$, 
for every 
 neighborhood $U$ of $x$ in $X$, and this means exactly that 
$x$ is a $D$-limit point of
 $\{ O_Z \mid Z \in S_ \mu ( \lambda ) \}$.
Thus, (2) is proved.

In order to prove (2) $\Rightarrow $  (1),
it is sufficient to prove that (2) implies Condition (4)
in Proposition \ref{equivo}.
 Let $\{ O_Z \mid Z \in S_ \mu ( \lambda ) \}$
be a  sequence  of nonempty open sets of $X$.
Letting
$C_ \alpha =
 \overline{  \bigcup  \{ O_Z \mid  Z \in S_ \mu ( \lambda ),  \alpha \in Z \}}$, for $\alpha \in \lambda $,
we need too show that
$\bigcap _{ \alpha \in \lambda } C_ \alpha \not= \emptyset $. 
Let $D$ be an ultrafilter as given by (2), and suppose that
$x$ is a $D$-limit point of $\{ O_Z \mid Z \in S_ \mu ( \lambda ) \}$.
We are going to show that $x \in \bigcap _{ \alpha \in \lambda } C_ \alpha $.
Suppose by contradiction that, for some $\alpha \in \lambda $, it happens that $x \not\in  C_ \alpha$. Since $C_\alpha$ is closed, 
$x$ has some neighborhood $U$ disjoint from $C_\alpha$.  
Notice that, if $Z \in S_ \mu ( \lambda )$  and $\alpha \in Z$,
then $C_ \alpha \supseteq O_Z $. Hence  
$\{ Z \in S_ \mu ( \lambda ) \mid    U \cap O_Z \not= \emptyset \}
\cap [ \alpha ) = \emptyset $,
hence $\{ Z \in S_ \mu ( \lambda ) \mid    U \cap O_Z \not= \emptyset \} \not \in D$,
since $D$ is an ultrafilter, and  $[ \alpha ) \in D$ by assumption, since $D$ 
is supposed to cover $\lambda$. But $\{ Z \in S_ \mu ( \lambda ) \mid    U \cap O_Z \not= \emptyset \} \not \in D$ contradicts the assumption that
$x$ is a $D$-limit point of
$\{ O_Z \mid Z \in S_ \mu ( \lambda ) \}$.
Hence $x \in \bigcap _{ \alpha \in \lambda } C_ \alpha $,
thus $\bigcap _{ \alpha \in \lambda } C_ \alpha \not= \emptyset $,
 and the proof is complete.  
\end{proof} 

\begin{remark} \label{caic} 
Theorem \ref{ufo} is inspired by results by X. Caicedo from his seminal paper \cite{C}. See also \cite{Cprepr}.
Caicedo proved  results similar to 
Theorem \ref{ufo} 
 for $ [ \mu, \lambda ]$-compactness. The result
analogous to the implication (1) $\Rightarrow $  (2) in Theorem \ref{ufo}
is Lemma 3.3 (i) in \cite{C}.
A common generalization and strengthening of both Theorem \ref{ufo} 
and \cite[Lemmata 3.1 and 3.2]{C} holds. See Theorem \ref{equivfdcpn} 
(1) $\Rightarrow $  (7) below.
 
Notice that, because of the well known result about
$ [ \mu, \lambda ]$-\brfrt compactness mentioned in Remark \ref{notk},
essentially all applications of results in \cite{C} can be obtained using only 
the particular case $\lambda= \mu$ of   \cite[Lemmata 3.1 and 3.2]{C}.
However, such a reduction is not   possible in the case of 
$\mathcal O$-$ [ \mu, \lambda ]$-compactness, by Remark \ref{gf}.
Hence it is necessary to deal with the more general case in which $\lambda \not= \mu$
is allowed.  
The idea from
\cite{Cprepr,C}
of treating the full general case
  is thus well-justified
%, at least in some different situation. 
\end{remark}

\begin{definition} \label{dpseudocp}   
If $D$ is an ultrafilter over $I$, then a topological space $X$ is said to be $D$-\emph{pseudocompact} (\cite{GS,GF})
 if and only if every  sequence 
$(O_i)_{i \in I}$ 
of nonempty open subsets of $X$
has some $D$-limit point in $X$. 
\end{definition}

\begin{definition} \label{rk}    
An ultrafilter $D$ over some set $I$ is
said to be \emph{$(\mu ,\lambda )$-regular} 
if and only if
there is  a function $f:I\to S_\mu(\lambda )$
such that
$\{i\in I|\alpha\in f(i) \}\in D $, for every $\alpha\in\lambda $.
See, e. g., \cite{mru} for equivalent definitions and for a survey of results on  $( \mu, \lambda  )$-regular ultrafilters.

If $D$ is an ultrafilter over $I$, 
and $f:I \to J$ is a function,
the ultrafilter $f(D)$ over $J$ is defined by
the following clause:
$Z \in f(D)$ if and only if $f ^{-1}(Z) \in D $.

With the above notation, it is trivial to see that
$D$ over $I$ is  $( \mu, \lambda  )$-regular
if and only if there exists some function
$f:I \to S_\mu(\lambda )$ such that 
$f(D)$ covers $\lambda$. 

In passing, let us mention that the above definitions involve the so-called
 Rudin-Keisler order. 
If $D$ and $E$ are two ultrafilters, respectively
over $I$ and $J$, then
$E$ is said to be less than or equal to $D$ in the \emph{Rudin-Keisler} 
(pre-) order, $E \leq _{\rk} D $ for short, if and only if
there exists some   function
$f:I \to J$ such that     
  $E= f(D)$.
If both  $E \leq _{\rk} D $
and  $D \leq _{\rk} E $, then
$E$ and $D$ are said to be
(Rudin-Keisler) \emph{equivalent}. 
 \end{definition}

The next fact  is trivial, but very useful.

 \begin{fact} \label{proj} 
If $D$ is an ultrafilter over $I$,
$X$ is a $D$-pseudocompact topological space, 
and $f:I\to J$ is a function, then
$X$ is $f(D)$-pseudocompact.
\end{fact}

\begin{corollary} \label{regimpcpn}%3.4
Suppose that  $D$ is a $( \mu, \lambda  )$-regular ultrafilter. 

If $X$ is a $D$-pseudocompact topological space,
then $X$ is
$\mathcal O$-$ [ \mu, \lambda ]$-compact.

More generally, if $(X_j) _{j \in J} $
is a  sequence  of  
$D$-pseudocompact topological spaces,
then the Tychonoff product 
$ \prod _{j \in J} X_j $  is
$\mathcal O$-$ [ \mu, \lambda ]$-compact.
 \end{corollary}

\begin{proof} 
By $( \mu, \lambda  )$-regularity, there is 
$f:I \to S_\mu(\lambda )$ such that 
$f(D)$ covers $\lambda$. 
By Fact \ref{proj}, $X$ is 
$f(D)$-pseudocompact, hence $\mathcal O$-$ [ \mu, \lambda ]$-compactness of $X$ 
follows from Theorem \ref{ufo} with $f(D)$ in place of $D$.
Notice that here $f(D)$ works ``uniformly'' for every  sequence ,
while, in the statement of Theorem \ref{ufo}(2), the ultrafilter, in general, depends on the  sequence.

The last statement follows from the known fact (\cite[Theorem 4.3]{GS})
that $D$-pseudocompactness is preserved under taking products.
\end{proof}  
 
A result analogous to Corollary \ref{regimpcpn} for
$[\mu, \lambda ]$-compactness is proved in \cite[Lemma 3.1]{C}.   

We now present a nice characterization of
$D$-pseudocompactness.

\begin{theorem} \label{equivodcpn} %3.5
Suppose that $D$ is an ultrafilter over some set $I$,
and $X$ is a topological space.
Then the following are equivalent.

\begin{enumerate} 
\item 
$X$ is $D$-pseudocompact.

\item
For every sequence 
$\{ O_i \mid i \in I \}$
of nonempty open sets of $X$, if, for 
$Z \in D$, we put
$C_Z= \overline{ \bigcup _{i \in Z} O_i} $, then
we have that 
$\bigcap _{Z \in D} C_Z \not= \emptyset  $.

\item
Whenever $(C_Z) _{Z \in D} $ is a  sequence  of closed sets of $X$
with the property that, for every $i \in I$, 
 $\bigcap _{i\in Z} C_Z $ contains some nonempty open set of $X$, then  
$\bigcap _{Z \in D} C_Z \not= \emptyset  $.

\item
For every open cover $(Q_Z) _{Z \in D} $ of $X$,
there is some $i \in I$ such that  
$\bigcup _{i\in Z} Q_Z $
is dense in $X$.

\item
For every sequence 
$\{ C_i \mid i \in I \}$
of closed sets of $X$, 
such that each
$C_i$ is properly contained in $X$, 
if, for 
$Z \in D$, we let 
$Q_Z$ be the interior of
$ \bigcap _{i \in Z} C_i $, then
we have that 
$ (Q_Z) _{Z \in D}$
is not a cover of $X$.
\end{enumerate}  
 \end{theorem}

 \begin{proof} 
(1) $\Rightarrow $  (2) 
By $D$-pseudocompactness, the sequence $\{ O_i \mid i \in I \}$  
has some $D$-limit point $x$ in $X$, that is,
$\{ i \in I \mid U \cap O_i \not= \emptyset \} \in D$, 
for every neighborhood $U$ of $x$ in $X$.

We are going to show that 
$ x \in \bigcap _{Z \in D} C_Z $.
Indeed, let $Z$ be any set in $D$. 
If $U$ is a neighborhood of $x$, then 
 $Z'=Z \cap \{ i \in I \mid U \cap O_i \not= \emptyset \} $ is still in $ D$, thus is
nonempty. Let $i \in Z'$. Then
$U \cap O_i \not= \emptyset$, and 
$C_Z \supseteq O_i$, since $i \in Z$. Hence
$U \cap C_Z \not= \emptyset$.
Since the above argument works for every neighborhood $U$ of
$x$, we have that
 $x \in C_Z$, since $C_Z$ is a closed set.

We have showed that 
$x \in C_Z$, for every
$Z \in D$,
hence 
$ x \in \bigcap _{Z \in D} C_Z $.

(2) $\Rightarrow $  (3)
For every $i \in I$, let $O_i$ be some nonempty open set of $X$ such that  
 $\bigcap _{i\in Z} C_Z \supseteq O_i$.
For every $Z \in D$,  put
$C'_Z= \overline{ \bigcup _{i \in Z} O_i} $.
By Clause (2), we have that 
$\bigcap _{Z \in D} C'_Z \not= \emptyset  $.
Since, for every $i \in Z$,
 $C_Z \supseteq O_i$,
we have that $C_Z \supseteq C'_Z$, for every $Z \in D$.   
Hence, 
$\bigcap _{Z \in D} C_Z \supseteq \bigcap _{Z \in D} C'_Z \not= \emptyset  $.

(3) $\Rightarrow $  (1) Suppose that $(O_i) _{i \in I} $
is a sequence of nonempty open sets of $X$.
For $Z \in D$, let 
$C_Z= \overline{ \bigcup _{i \in Z}  O_i} $.  
Hence, for every $i \in Z$,
$C_Z \supseteq   O_i$, and, for every $i \in I$, 
 $\bigcap _{i\in Z} C_Z $ contains the nonempty open set $  O_i$. 
 
By (3), there is
some $x \in X$
such that  
$ x \in \bigcap _{Z \in D} C_Z$.
It is enough to show that $x$ is a $D$-limit point of  $(O_i) _{i \in I} $.
If not, $x$ has some neighborhood $U$ such that  
$\{ i \in I \mid U \cap O_i \not= \emptyset \} \not\in D$, that is,
 $\{ i \in I \mid U \cap O_i = \emptyset \} \in D$. 
Letting $Z=\{ i \in I \mid U \cap O_i = \emptyset \}$,
we have that $U \cap \bigcup _{i \in Z} O_i = \emptyset  $,
but this contradicts 
 $ x \in C_Z= \overline{ \bigcup _{i \in Z}  O_i} $.

(3) $\Leftrightarrow $  (4) and
(2) $\Leftrightarrow $  (5) are obtained by considering complements.
\end{proof}

\section{Theorems about products} \label{prod} 

In this section we consider, for a product space $ \prod _{j \in J} X_j$, 
a variant of $\mathcal O $-$ [ \mu, \lambda ]$-compactness,
a variant which takes into account 
all the open sets in  the box topology   on
the set $ \prod _{j \in J} X_j$.
This notion shall be used in order to provide
 a characterization of all 
those spaces $X$ which are $D$-pseudocompact,
for some $( \mu, \lambda  )$-regular ultrafilter $D$
(Theorem \ref{fprodo}).

We shall need to consider the \emph{set}  $ \prod _{j \in J} X_j$
endowed  both with the Tychonoff topology and with the 
\emph{box} topology. A base for the latter topology is given by \emph{all}
the products  $ \prod _{j \in J} O_j$, each $ O_j$ being an open set 
of $X_j$. When we write $ \prod _{j \in J} X_j$, we shall always assume that the product is endowed with the Tychonoff topology, while 
$ \bigbox _{j \in J} X_j$ shall denote the product endowed with the box topology.

\begin{definition} \label{fcpnopbox}
Suppose that $(X_j) _{j\in J} $
is a  sequence  of topological spaces. 
We say that the topological space  $ \prod _{j \in J} X_j$ is
$\mathcal O ^{\Box} $-$ [ \mu, \lambda ]$-\emph{compact}
if and only if the following holds.

For every  sequence  $( C _ \alpha ) _{ \alpha \in \lambda } $
 of closed sets of $ \prod _{j \in J} X_j$, if,
for every $Z \subseteq \lambda $ with $ |Z|< \mu$,
there exists a nonempty open set $O_Z$ of $ \bigbox _{j \in J} X_j$   such that   
$ \bigcap _{ \alpha \in Z}  C_ \alpha \supseteq O_Z$,
then  
$ \bigcap _{ \alpha \in \lambda }  C_ \alpha \not= \emptyset $.
 \end{definition}   

Notice that 
$\mathcal O ^{\Box} $-$ [ \mu, \lambda ]$-compactness
is a notion stronger than 
$\mathcal O $-$ [ \mu, \lambda ]$-compactness,
that is,
every 
$\mathcal O ^{\Box} $-$ [ \mu, \lambda ]$-compact
product $ \prod _{j \in J} X_j$ is
$\mathcal O  $-$ [ \mu, \lambda ]$-compact.
The two notions are distinct, in general, as we shall see in 
Remark \ref{civuolebox}.
Notice also that every 
$ [ \mu, \lambda ]$-compact product is
$\mathcal O^{\Box}$-$ [ \mu, \lambda ]$-compact.

\begin{remark} \label{nothom} 
Notice that $\mathcal O ^{\Box} $-$ [ \mu, \lambda ]$-compactness is not
an intrinsic  property of the topological space
$Y= \prod _{j \in J} X_j$. That is,
$\mathcal O ^{\Box} $-$ [ \mu, \lambda ]$-\brfrt compactness does not only depend
on the topology on $Y$, but depends also on the way
$Y$ is realized as a product.
There might be two homeomorphic spaces, say,  
$Y= \prod _{j \in J} X_j$ and
$Z= \prod _{h\in H} Y_h$
such that $Y$, \emph{as a product}
$ \prod _{j \in J} X_j$, is
$\mathcal O ^{\Box} $-$ [ \mu, \lambda ]$-compact,
while 
$Z $, \emph{as a product} $  \prod _{h\in H} Y_h$, is not.
Just to consider a simple case, if 
$Y=  \prod _{j \in J} X_j$, and $Z$ is a homeomorphic copy
of $Y$, and we consider $Z$ ``as itself'', that is, as the product of just a single factor, then
 $Z$ is $\mathcal O ^{\Box} $-$ [ \mu, \lambda ]$-compact
if and only if it is 
$\mathcal O  $-$ [ \mu, \lambda ]$-compact.
On the contrary, as we shall see, 
$\mathcal O ^{\Box} $-$ [ \mu, \lambda ]$-compactness
and 
$\mathcal O  $-$ [ \mu, \lambda ]$-compactness  
are distinct notions, in general.

The above remark will cause no problem here, since we will
always be dealing with a space $Y=\prod _{j \in J} X_j$ together with
just one single realization of $Y $ as $  \prod _{j \in J} X_j$.
In other words, we shall never deal with the homeomorphism
equivalence class of
$Y $,
but we shall always deal with $Y= \prod _{j \in J} X_j$ 
just in its concrete realization.
\end{remark}   

Of course, $\mathcal O ^{\Box} $-$ [ \mu, \lambda ]$-compactness
can be characterized in a way similar to the characterizations
of  
$\mathcal O  $-$ [ \mu, \lambda ]$-compactness 
given in Proposition \ref{equivo}.
Clause (7) in the next proposition is proved as  the last
statement of Definition \ref{fcpnop}.

\begin{proposition} \label{equivobox}%4.3
For every  sequence  $ (X_j)_{j \in J}$
 of topological spaces, and $\lambda$, $\mu$  infinite cardinals, 
the following are equivalent, where, in items (2)-(5), closures are computed in
$ \prod _{j \in J} X_j$.

\begin{enumerate} 
\item
$\prod _{j \in J} X_j$ is
$\mathcal O^{\Box}$-$ [ \mu, \lambda ]$-compact.

\item
For every  sequence  $( P _ \alpha ) _{ \alpha \in \lambda } $ of 
 subsets of $ \prod _{j \in J} X_j$, if,
for every $Z \subseteq \lambda $ with $ |Z|< \mu$,
there exists a nonempty open set $O_Z$ of $ \bigbox _{j \in J} X_j$   such that   
$ \bigcap _{ \alpha \in Z}  P_ \alpha \supseteq O_Z$,
then  
$ \bigcap _{ \alpha \in \lambda }  \overline{P}_ \alpha \not= \emptyset $.

\item
For every  sequence  $( Q _ \alpha ) _{ \alpha \in \lambda } $
 of open sets of $ \bigbox _{j \in J} X_j$, if,
for every $Z \subseteq \lambda $ with $ |Z|< \mu$,
there exists a nonempty open set $O_Z$ of $ \bigbox _{j \in J} X_j$   such that   
$ \bigcap _{ \alpha \in Z}  Q_ \alpha \supseteq O_Z$,
then  
$ \bigcap _{ \alpha \in \lambda }  \overline{Q}_ \alpha \not= \emptyset $.

\item
For every  sequence 
$\{ O_Z \mid Z \in S_ \mu ( \lambda ) \}$
of nonempty open sets of $ \bigbox _{j \in J} X_j$,
it happens that
    $ \bigcap _{ \alpha \in \lambda } 
 \overline{  \bigcup  \{ O_Z \mid  Z \in S_ \mu ( \lambda ),  \alpha \in Z \}}
  \not= \emptyset $.

\item
For every  sequence 
$\{ O_Z \mid Z \in S_ \mu ( \lambda ) \}$
of nonempty open sets of $ \bigbox _{j \in J} X_j$,
the following holds.
If, for every finite subset $W$  of $\lambda$,
we put  
$Q_W= 
  \bigcup \{ O_Z \mid Z \in S_ \mu ( \lambda ) \text{ and } Z \supseteq
W\}$,
then 
$\bigcap \{ \overline{Q_W} \mid W \text{ is a finite subset of } \lambda  \}\not= \emptyset $.

\item
For every  sequence 
$\{ C_Z \mid Z \in S_ \mu ( \lambda ) \}$
of closed sets of $ \bigbox _{j \in J} X_j$, such that each
$C_Z$ is properly contained in $X$,
if we let, for $\alpha \in \lambda $,
$P_ \alpha $ be the interior
(computed in
$ \prod _{j \in J} X_j$)
of 
$ \bigcap  \{ C_Z \mid  Z \in S_ \mu ( \lambda ),  \alpha \in Z \}$, then
we have that 
$(P_ \alpha ) _{ \alpha \in \lambda } $
is not a cover of $X$.  

\item
For every $\lambda$-indexed open cover
 $( Q _ \alpha ) _{ \alpha \in \lambda } $
 of $ \prod _{j \in J} X_j$, there exists $Z \subseteq \lambda $, with $ |Z|< \mu $,
such that
$ \bigcup _{ \alpha \in Z}  Q_ \alpha$
is a dense subset in $ \bigbox _{j \in J} X_j$.
\end{enumerate}
 \end{proposition}

The proof of Theorem \ref{ufo}
carries over essentially unchanged in order to get the following
useful theorem.

\begin{theorem} \label{ufobox}%4.4
For every  sequence  $ (X_j)_{j \in J}$
 of topological spaces, and $\lambda$, $\mu$  infinite cardinals, 
the following are equivalent.

\begin{enumerate} 
\item
$\prod _{j \in J} X_j$ is
$\mathcal O^{\Box}$-$ [ \mu, \lambda ]$-compact.

\item
For every  sequence 
$\{ O_Z \mid Z \in S_ \mu ( \lambda ) \}$
of nonempty open sets of $ \bigbox _{j \in J} X_j$,
there exists an ultrafilter $D$ over 
$S _\mu ( \lambda )$ which covers $\lambda$ and 
such that 
 $\{ O_Z \mid Z \in S_ \mu ( \lambda ) \}$
has a $D$-limit point in $\prod _{j \in J} X_j$.
\end{enumerate}
 \end{theorem} 

Theorem \ref{ufobox}
can be used to improve
the last statement in Corollary \ref{regimpcpn}.

\begin{corollary} \label{regimpcpnbox}
Suppose that  $D$ is a $( \mu, \lambda  )$-regular ultrafilter. 

If $(X_j) _{j \in J} $
is a  sequence  of  
$D$-pseudocompact topological spaces,
then  
$ \prod _{j \in J} X_j $  is
$\mathcal O^{\Box}$-$ [ \mu, \lambda ]$-compact.
 \end{corollary} 
 
We are now going to show that a topological space $X$ 
is $D$-\brfrt pseudocompact for some $( \mu, \lambda  )$-regular ultrafilter $D$ 
if and only if all (Tychonoff) powers of $X$ are $\mathcal O^{\Box}$-$ [ \mu, \lambda ]$-compact. We shall denote by $X ^ \delta $
the Tychonoff product of $\delta$-many copies of $X$.

\begin{theorem} \label{fprodo} %4.6
For every 
 topological space $X$, and $\lambda$, $\mu$  infinite cardinals, 
the following are equivalent. 
\begin{enumerate} 
\item
There exists some ultrafilter 
 $D$ over  
$ S_ \mu ( \lambda )$ which covers $\lambda$, 
and such that $X$ is $D$-pseudocompact.
\item 
There exists some $( \mu, \lambda  )$-regular ultrafilter $D$
(over any set)
such that $X$ is $D$-pseudocompact.
\item 
There exists some $( \mu, \lambda  )$-regular ultrafilter $D$
such that, for every cardinal $\delta$,
the space $X^ \delta $  is $D$-pseudocompact.
\item
The power $X^ \delta $ is 
$\mathcal O^{\Box}$-$ [ \mu, \lambda ]$-compact,
for every cardinal $\delta$.
\item
The power $X^ \delta $ is 
$\mathcal O^{\Box}$-$ [ \mu, \lambda ]$-compact,
for  $\delta= \min \{ 2 ^{2^ \kappa  }, (w(X))^ \kappa  \}$,
where $\kappa=\lambda ^{<\mu}$.  
\end{enumerate}   
\end{theorem}  

\begin{proof}
(1) $\Rightarrow $  (2) is trivial, since if 
 $D$ is over  
$ S_ \mu ( \lambda )$ and covers $\lambda$, 
then $D$ is  $( \mu, \lambda  )$-regular.

(2) $\Rightarrow $  (3) follows from the mentioned result from \cite[Theorem 4.3]{GS},
asserting that a product of $D$-pseudocompact spaces is still 
$D$-pseudocompact.

(3) $\Rightarrow $  (4) follows from Corollary \ref{regimpcpnbox}.

(4) $\Rightarrow $  (5) is trivial.

(5) $\Rightarrow $  (1) 
We first 
consider the case $\delta =(w(X))^ \kappa $.

Let $\mathcal B$ be a base of $X$ of cardinality $w(X)$.
 Thus, there are $ \delta $-many
 $S _ \mu (\lambda)$-indexed
sequences of elements of $\mathcal B$,
since $|S _ \mu (\lambda)|= \kappa $.
Let us enumerate
them as 
 $\{  Q _{\beta , Z} \mid Z \in S_ \mu ( \lambda ) \}$,
$ \beta $ varying in  $ \delta $.
In $X^ \delta $
consider the sequence 
 $\{ \prod _{ \beta \in \delta } Q _{\beta , Z} \mid Z \in S_ \mu ( \lambda ) \}$.
For every $Z \in S _ \mu (\lambda)$,  
the set $\prod _{ \beta \in \delta } Q _{\beta , Z} $
is open in the box topology on $X^ \delta $.
By the $\mathcal O^{\Box}$-$ [ \mu, \lambda ]$-compactness of
$X^ \delta $, and by Theorem  \ref{ufobox}(1) $\Rightarrow $  (2), 
there exists an ultrafilter $D$ over 
$S _\mu ( \lambda )$ which covers $\lambda$ and 
such that 
 $\{ \prod _{ \beta \in \delta } Q _{\beta , Z} \mid Z \in S_ \mu ( \lambda ) \}$
has some $D$-limit point $x$ in $X ^ \delta $.

We are going to show that $X$ is $D$-pseudocompact.  
So, let 
 $\{ O_Z \mid Z \in S_ \mu ( \lambda ) \}$
be a sequence of nonempty open sets of $X$.
Since $\mathcal B$ is a base for $X$, then,
for every $Z \in S_ \mu ( \lambda )$,
there is a nonempty $B_Z$ in $\mathcal B$
such that $O_Z \supseteq B_Z$.
Choose one such $B_Z$ for each  $Z \in S_ \mu ( \lambda )$.     
The sequence 
 $\{ B_Z \mid Z \in S_ \mu ( \lambda ) \}$
 is an
 $S _ \mu (\lambda)$-indexed
sequences of elements of $\mathcal B$.
Since, by construction, all such sequences are enumerated by
$\{  Q _{\beta , Z} \mid Z \in S_ \mu ( \lambda ) \}$,
there is some
$\beta _0 \in \delta $ such that 
$B_Z=Q _{\beta _0 , Z} $, for every 
$Z \in S _ \mu (\lambda) $.  

By what we have proved before, the sequence
 $\{ \prod _{ \beta \in \delta } Q _{\beta , Z} \mid Z \in S_ \mu ( \lambda ) \}$
has some $D$-limit point $x$ in $X ^ \delta $, say 
$x=(x_ \beta ) _{ \beta \in \delta } $.
A trivial property of $D$-limits
implies that, for every $\beta \in \delta $,
we have that $x_\beta$ is a $D$-limit of   
$ \{ Q _{\beta , Z} \mid Z \in S_ \mu ( \lambda ) \}$.
In particular, by taking $\beta= \beta _0$, we
get that $x _{\beta_0}$ is a $D$-limit point of  
$\{ B_Z \mid Z \in S_ \mu ( \lambda ) \}$.

Since $O_Z \supseteq B_Z$, for every 
$ Z \in S _ \mu (\lambda)$,
we get that  
$x _{\beta_0}$ is also a $D$-limit point of  
 $\{ O_Z \mid Z \in S_ \mu ( \lambda ) \}$.
We have proved that every 
sequence
 $\{ O_Z \mid Z \in S_ \mu ( \lambda ) \}$
of nonempty open sets of $X$
has some $D$-limit point in $X$, that is, 
$X$ is $D$-pseudocompact.

Now we consider the case 
$\delta= 2 ^{2^ \kappa  } $.
 We shall prove that
if   $\delta= 2 ^{2^ \kappa  } $ and
(1) fails, then  (5) fails.
If (1) fails, then, for every ultrafilter $D$ over
$S_ \mu ( \lambda )$ which covers $\lambda$,  
there is a sequence 
 $\{ O_Z \mid Z \in S_ \mu ( \lambda ) \}$
of nonempty open sets of $X$
 which has no $D$-limit point.
Since there are $\delta $-many ultrafilters over $S_ \mu ( \lambda )$,
we can enumerate the above sequences as
 $\{ O _{\beta , Z} \mid Z \in S_ \mu ( \lambda ) \}$,
$ \beta $ varying in  $ \delta $.
 Now, given any ultrafilter 
$D$ over
$S_ \mu ( \lambda )$ and covering $\lambda$,
it is not the case that 
 the sequence 
 $\{ \prod _{ \beta \in \delta } O _{\beta , Z} \mid Z \in S_ \mu ( \lambda ) \}$
has some $D$-limit point.
Indeed, were $x=(x_ \beta ) _{ \beta \in \delta } $
a $D$-limit point of 
 $\{ \prod _{ \beta \in \delta } O _{\beta , Z} \mid Z \in S_ \mu ( \lambda ) \}$,
then, by a trivial property of $D$-limits,
for every $\beta \in \delta $,
$x_ \beta $ would be a $D$-limit point of
  $\{  O _{\beta , Z} \mid Z \in S_ \mu ( \lambda ) \}$.
This is a contradiction since, by construction, 
for every ultrafilter $D$ over $S_ \mu ( \lambda )$  covering $\lambda$,
there exists some 
$\beta \in \delta $ such that 
  $\{  O _{\beta , Z} \mid Z \in S_ \mu ( \lambda ) \}$
has no $D$-limit point.

We have showed that
for no ultrafilter 
$D$ over
$S_ \mu ( \lambda )$ and covering $\lambda$ 
 the sequence 
 $\{ \prod _{ \beta \in \delta } O _{\beta , Z} \mid Z \in S_ \mu ( \lambda ) \}$
has some $D$-limit point. 
Since, for every $Z \in S_ \mu ( \lambda )$,  $\prod _{ \beta \in \delta } O _{\beta , Z}$ is an 
open set of the box topology on $X^ \delta $, we get that,  
by Theorem  \ref{ufobox}
(1) $\Rightarrow $  (2),
$X^ \delta $ is not
$\mathcal O^{\Box}$-$ [ \mu, \lambda ]$-compact, that is, 
(5) fails.
\end{proof}

\begin{remark} \label{cofslm} %4.7
Condition (5) in Theorem \ref{fprodo} can be improved
to the effect that we can take 
$\kappa$ there to be equal to the cofinality 
of the partial order 
$S_ \mu ( \lambda )$.
A subset $H$ of $S_ \mu ( \lambda )$
is said to be \emph{cofinal} in $S_ \mu ( \lambda )$ 
if and only if, for every 
$Z \in S_ \mu ( \lambda )$,
 there is $Z' \in H$ such that 
$Z \subseteq Z'$.
The \emph{cofinality} $\cf S_ \mu ( \lambda )$ of $S_ \mu ( \lambda )$
is the minimal cardinality of some subset    
$H$ cofinal in $S_ \mu ( \lambda )$.
Notice that if $\lambda$ is regular, then 
$ \cf S_ \lambda  ( \lambda )= \lambda $ and, more generally,
$ \cf S_ \lambda  ( \lambda^+ )= \lambda^+ $.
Highly non trivial results about 
$\cf S_ \mu ( \lambda )$ are consequences of Shelah's pcf-theory
\cite{Sh}.

For the rest of this remark, let us fix some 
subset    
$H$ cofinal in $S_ \mu ( \lambda )$.

All the definitions and results involving  
$S_ \mu ( \lambda )$ can be modified in order to apply
to $H$, too. 
In particular, in the definitions of
$\mathcal O$-$ [ \mu, \lambda ]$-compactness and of 
$\mathcal O^{\Box}$-$ [ \mu, \lambda ]$-compactness,
we get an equivalent notion if we consider only
those $Z \in H$.  
Similarly, in Propositions \ref{equivo} and \ref{equivobox}  
we can equivalently consider $H$-indexed sequences, rather than
$S_ \mu ( \lambda )$-indexed sequences, that is,
we can replace everywhere $Z \in S_ \mu ( \lambda )$
by $Z \in H$, still obtaining the results.

Moreover, we can say that an ultrafilter $D$ over  $H$
covers $\lambda$ if and only if,
for every $\alpha \in \lambda $,
it happens that
$[ \alpha )_H =\{ Z \in H \mid \alpha \in Z \} \in D$.
With this definition,
we have that Theorems \ref{ufo} and \ref{ufobox},  
too, hold, if $Z \in S_ \mu ( \lambda )$
is everywhere replaced by $Z \in H$.

Moreover, let $f:S_ \mu ( \lambda ) \to H$
be defined in such a way that 
$Z \subseteq f(Z) $. 
If $D$ is over $ S_ \mu ( \lambda )$
and covers $\lambda$, then $f(D)$ is over
(a subset of) $H$, and   $f(D)$, too, covers $\lambda$.
The above observations give us the possibility of proving
Theorem \ref{fprodo} with the improved value
$\kappa= \cf S_ \mu ( \lambda )$
in Condition (5).
 \end{remark}

\begin{remark} \label{civuolebox}  
In order to get results like Theorem \ref{fprodo},
it is actually necessary to deal with 
$\mathcal O^{\Box}$-$ [ \mu, \lambda ]$-compactness,
rather than with $\mathcal O$-$ [ \mu, \lambda ]$-compactness.
Indeed, \cite[Example 4.4]{GS} 
constructed
 a Tychonoff space $X$ such that  
all  powers of $X$ are pseudocompact
but for no ultrafilter $D$ uniform over $ \omega$,
$X$ is $D$-pseudocompact. 
By Remark \ref{precedenti}, all  powers of $X$ are $\mathcal O$-$ [ \omega, \omega ]$-compact.
 The condition that, for no ultrafilter $D$ uniform over $ \omega$,
$X$ is $D$-pseudocompact is easily seen to be equivalent to the property that 
for no ultrafilter $D$  over $ S_ \omega ( \omega )$ and covering $ \omega$,
$X$ is $D$-pseudocompact.
The equivalence can be proved directly; otherwise, notice that, for 
$\mu=\lambda$ a regular cardinal,
 Condition (4) in Theorem \ref{fprodo} coincides with 
Condition (5) in \cite[Corollary 5.5]{LiF}, hence 
the respective Conditions (1) are equivalent. 

Since, for no ultrafilter $D$  over $ S_ \omega ( \omega )$ and covering $ \omega$,
$X$ is $D$-pseudocompact, we get, by Theorem \ref{fprodo},
that not every power of $X$ is 
$\mathcal O^{\Box}$-$ [ \omega, \omega ]$-compact,
but, as we remarked, every power of $X$ is
 $\mathcal O$-$ [ \omega, \omega ]$-compact, thus the two notions are distinct, in general.
Indeed, by 
 Remark 
\ref{cofslm},
we have that
$X^ \delta  $ is not $\mathcal O^{\Box}$-$ [ \omega, \omega ]$-compact
for $ \delta = 2 ^{2^ \omega } $. 

  In particular, Conditions (4) and (5) 
in Theorem \ref{fprodo}
are in general not equivalent to the other conditions, if we replace
  $\mathcal O^{\Box}$-$ [ \mu, \lambda  ]$-compactness with $\mathcal O$-$ [  \mu, \lambda   ]$-compactness.
 \end{remark} 

Indeed, as is the case for pseudocompactness, we can show that the $\mathcal O$-$ [ \mu, \lambda ]$-compactness
of a product depends only on 
the $\mathcal O$-$ [ \mu, \lambda ]$-compactness of all subproducts of some small number of factors. Thus,
we have an analogue for $\mathcal O$-$ [ \mu, \lambda ]$-compactness
of the equivalence (4) $\Leftrightarrow $  (5) in Theorem \ref{fprodo}.

\begin{lemma} \label{lm} 
If $X$ and $Y$ are topological spaces,
 $f:X \to Y$ is a continuous and surjective function, and $X$ is   
$\mathcal O$-$ [ \mu, \lambda ]$-compact then also
$Y$ is $\mathcal O$-$ [ \mu, \lambda ]$-compact. 
\end{lemma} 

\begin{theorem} \label{prodnorm}   
Suppose that $(X_j) _{j \in J} $
is a  sequence  of  topological spaces.
Then the product 
$ \prod _{j \in J} X_j $  is
$\mathcal O$-$ [ \mu, \lambda ]$-compact
if and only if any subproduct of $ \leq \kappa $
factors is $\mathcal O$-$ [ \mu, \lambda ]$-compact,
where $\kappa=\lambda ^{<\mu}$.
Indeed, the result can be improved to 
$\kappa= \cf S_ \mu ( \lambda )$. 
\end{theorem}

 \begin{proof}
The only-if part is immediate from Lemma \ref{lm}.
 
To prove the converse, given $(C_ \alpha ) _{ \alpha \in \lambda } $ as 
in the definition of $\mathcal O $-$ [ \mu, \lambda ]$-compactness,
we might assume, without loss of generality, that the $O_Z$'s are members of the canonical base of
 $ \prod _{j \in J} X_j $, that is,
each $O_Z$ has the form $ \prod _{j \in J} Q_j $,
 where each $Q_j$ is an open set of 
$X_j$, and $Q_j=X_j$,  for all $j \in J \setminus J_Z $,
for some finite $J_Z \subseteq J $. 

If $J'= \bigcup _{Z \in S_ \mu ( \lambda )} J_Z$,
and $\pi : \prod _{j \in J} X_j  \to \prod _{j \in J'} X_j $ 
is the canonical projection, then,
by assumption,
 $\prod _{j \in J'} X_j $ is
$\mathcal O $-$ [ \mu, \lambda ]$-compact,
since $|J'|\leq \kappa $, 
hence
$\bigcap_{ \alpha \in \lambda } \pi(C_ \alpha)  \not= \emptyset    $,
and this clearly implies
$\bigcap_{ \alpha \in \lambda }  C_ \alpha   \not= \emptyset    $.
 
By arguments similar to those in Remark \ref{cofslm},
we can improve the value of $\kappa$ to $\cf S_ \mu ( \lambda )$.
\end{proof}

For sake of simplicity, in the statement of Theorem \ref{fprodo} we have dealt
with a single topological space $X$. However, a version of the theorem
holds for families of topological spaces.

\begin{theorem} \label{fprodot}
For every family $T$ of
 topological spaces, and $\lambda$, $\mu$  infinite cardinals, 
the following are equivalent. 
\begin{enumerate} 
\item 
There exists some $( \mu, \lambda  )$-regular ultrafilter $D$
(which can be taken over $ S_ \mu ( \lambda )$)
such that, for every $X \in T$, we have that  $X$ is $D$-pseudocompact.
\item
Every product of any number of members of $T$ (allowing repetitions) is 
$\mathcal O^{\Box}$-$ [ \mu, \lambda ]$-compact.
\item
Every product of members of $T$ (allowing repetitions) 
with at most $\delta$ factors is 
$\mathcal O^{\Box}$-$ [ \mu, \lambda ]$-compact, where
  $\delta= \min \{ 2 ^{2^ \kappa  },
\sup\{|T|, \nu \}\}$,
for $\nu=\sup _{X \in T}  (w(X))^ \kappa $ and 
$\kappa=\lambda ^{<\mu}$  (indeed, this can be improved to $\kappa= \cf S_ \mu ( \lambda )$).  
\end{enumerate}   
\end{theorem}

\begin{corollary} \label{transferiffreg}
For $\mu$, $\lambda$, $\mu'$ and $\lambda'$ infinite cardinals, the following are equivalent.

(a) Every $( \mu, \lambda )$-regular ultrafilter
is $( \mu', \lambda' )$-regular.

(b) For every family $T$ of topological spaces, 
if every product of any number of members of $T$ (allowing repetitions) is 
$\mathcal O^{\Box}$-$ [ \mu, \lambda ]$-compact, then
every product of any number of members of $T$ (allowing repetitions) is 
$\mathcal O^{\Box}$-$ [ \mu', \lambda' ]$-compact.

(c) For every topological space $X$, 
if every power of $X$ is
$\mathcal O^{\Box}$-$ [ \mu, \lambda ]$-compact,
then every power of $X$ is
$\mathcal O^{\Box}$-$ [ \mu', \lambda' ]$-compact.

(d) Same as (c), restricted to Tychonoff spaces.
\end{corollary} 

\begin{proof}
(a) $\Rightarrow $  (b) 
Suppose that the assumption in (b) holds.
By Theorem \ref{fprodot} (2) $\Rightarrow $  (1),
there exists some 
$( \mu, \lambda  )$-regular ultrafilter $D$
 such that, for every $X \in T$, we have that  $X$ is $D$-pseudocompact.
By (a), $D$ is $( \mu', \lambda'  )$-regular.
Hence, by Theorem \ref{fprodot} (1) $\Rightarrow $  (2),
every product of any number of members of $T$  is 
$\mathcal O^{\Box}$-$ [ \mu', \lambda' ]$-compact.

(b) $\Rightarrow $  (c) and (c) $\Rightarrow $  (d) are trivial.

(d) $\Rightarrow $  (a) Garcia-Ferreira \cite{GF} constructs, for every ultrafilter $D$,
say over $I$, 
a Tychonoff space $P _{RK}(D) $ such that, for every ultrafilter $E$, say over $J$,
the space $P _{RK}(D) $ is $E$-pseudocompact if and only if $E=f(D)$ 
for some function $f:I \to J$, that is
if and only if 
$E \leq_{\rk} D$ in the Rudin-Keisler order. 

Let $D$ be a $( \mu, \lambda )$-regular ultrafilter, say over $I$. 
By above,  $X=P _{RK}(D) $ is $D$-pseudocompact, hence,
by Theorem \ref{fprodo} (2) $\Rightarrow $  (4),
every power of  $X$ is 
$\mathcal O^{\Box}$-$ [ \mu, \lambda ]$-compact.

By (d),
every power of  $X$ is 
$\mathcal O^{\Box}$-$ [ \mu', \lambda' ]$-compact and,
by Theorem \ref{fprodo} (2) $\Rightarrow $  (4),
$X$ is $E$-pseudocompact, for some
$( \mu', \lambda' )$-regular ultrafilter $E$ over some set $J$.
By the above-mentioned result from \cite{GF}, 
$E=f(D)$, 
for some function $f: I \to J$.
By a trivial property of the Rudin-Keisler order, 
$D$ is $( \mu', \lambda' )$-regular,
thus (a) is proved.
\end{proof}

Many results are known about 
cardinals for which Condition (a)
in Corollary \ref{transferiffreg} holds.
See \cite{mru} for a survey. 
Corollary \ref{transferiffreg}
can be applied in each of these cases.

\begin{remark} \label{gf}
As we mentioned in Remark \ref{notk},
 $ [ \mu, \lambda ]$-compactness is equivalent to 
 $ [ \kappa , \kappa  ]$-compactness  for every $\kappa$ such that 
$\mu \leq \kappa \leq \lambda $. 
We now show that the analogous result fails, in general, for 
 $\mathcal O$-$ [ \mu, \lambda ]$-compactness.

Under some set-theoretical assumption, 
\cite{Ka} constructed an ultrafilter $D$ uniform over $ \omega_1$
and an ultrafilter $D'$ over $ \omega$ such that, for every ultrafilter $E$,   
it happens that  $E \leq _{\rk}  D$ if and only if  $E$ is Rudin-Keisler equivalent
either to $D$ or to $D'$.   
By the results from \cite{GF}  mentioned in the proof 
of Corollary \ref{transferiffreg},
the space $P _{RK}(D) $ is
both $D$-pseudocompact  and 
$D'$-pseudocompact,  hence  
both $\mathcal O$-$ [ \omega , \omega  ]$-compact and
$\mathcal O$-$ [ \omega_1 , \omega_1  ]$-compact,
since every uniform ultrafilter over some cardinal 
$\lambda$ is $( \lambda , \lambda  )$-regular
(see, e. g., \cite{mru}).
Indeed, by Corollary \ref{regimpcpnbox}, 
all powers of $P _{RK}(D) $ are even
both $\mathcal O^{\Box}$-$ [ \omega , \omega  ]$-compact and
$\mathcal O^{\Box}$-$ [ \omega_1 , \omega_1  ]$-compact.

However, \cite{GF} proved that  
$P _{RK}(D) $ is not even
$ \omega _1$-pseudocompact.
Since, by \cite[Theorem 2(c)]{Re},  every   
$\mathcal O$-$ [ \omega , \lambda  ]$-compact Tychonoff 
space is $\lambda$-\brfrt pseudocompact, 
we have that 
$P _{RK}(D) $ is not
$\mathcal O$-$ [ \omega , \omega_1  ]$-compact
($\mathcal O$-$ [ \omega , \lambda  ]$-compact spaces
are called weakly-initially compact in \cite{Re}).
 \end{remark}

\section{The abstract framework} \label{abstrfr} 

In this final section we mention that our results actually hold
in the general framework introduced in \cite{LiF}.
 In \cite{LiF} each compactness property is defined relative to some
family $\mathcal F$ of subsets of a topological space $X$.
By taking $\mathcal F$ to be either the set of all singletons of $X$, or
 the set of all nonempty open sets of $X$,
this generalized  approach provides
a unified treatment of definitions and results
about $ [ \mu, \lambda ]$-compactness
and related compactness notions, on one side, 
and about
$\mathcal O$-$ [ \mu, \lambda ]$-compactness
and related pseudocompactness-like notions,
on the other side.

In the case of  $ [ \mu, \lambda ]$-compactness, 
as we shall point after each single result,
most of the theorems we get are known;
in the case when $\mathcal F= \mathcal O$ 
we usually get back the results obtained in the previous sections.

\begin{definition} \label{f} %5.1
The definitions of $\mathcal F$-$ [ \mu, \lambda ]$-compactness and of $\mathcal F$-$D$-compactness
  can be obtained, respectively, from the 
definitions of 
$\mathcal O$-$ [ \mu, \lambda ]$-compactness (Definition \ref{fcpnop})
and of $D$-pseudocompactness (Definition \ref{dpseudocp}), by replacing the family $\mathcal O$  off all nonempty
open sets with another  specified  family $\mathcal F$ of subsets of $X$.

In more detail, let $X$ be a topological space,
and let $\mathcal F$ be any family of subsets of $X$.

Let $\lambda$ and $\mu$ be infinite cardinals.
We say that $X$ is
$\mathcal F$-$ [ \mu, \lambda ]$-\emph{compact}
if and only if,
for every sequence $( C _ \alpha ) _{ \alpha \in \lambda } $
 of closed sets of $X$, if,
for every $Z \subseteq \lambda $ with $ |Z|< \mu$,
there exists $F \in \mathcal F$ such that   
$ \bigcap _{ \alpha \in Z}  C_ \alpha \supseteq F$,
then  
$ \bigcap _{ \alpha \in \lambda }  C_ \alpha \not= \emptyset $.

Let $D$ be an ultrafilter over some set $Z$.
We say that $X$ is  $\mathcal F$-$D$-\emph{compact} if and only if every sequence
$(F_z)_{z \in Z}$ 
of members of $\mathcal F$ 
has some $D$-limit point in $X$.

When, in the preceding definitions, 
$\mathcal F= \mathcal O$, the family of all the nonempty open sets 
of $X$, we get back Definitions \ref{fcpnop}
and \ref{dpseudocp}.
When
$\mathcal F$ is taken to be the family of all singletons of $X$, 
we get back the more familiar notions of, respectively, $ [ \mu, \lambda ]$-compactness and of $D$-compactness.
See \cite{LiF} for more information. 
In particular, notice that, for $\mu= \lambda $
a regular cardinal,  \cite{LiF} provides a very refined theory of
$\mathcal F$-$ [ \lambda , \lambda ]$-compactness.
In the particular case $\mu= \lambda $ regular, the results presented
in  \cite{LiF} are usually stronger than the results presented here
for 
$\mathcal F$-$ [ \mu , \lambda ]$-compactness.
Notice also that,  by Remark \ref{gf}, the theory of 
$\mathcal F$-$ [ \mu , \lambda ]$-compactness, in general, cannot be ``reduced'' to the theory
of $\mathcal F$-$ [ \kappa  , \kappa  ]$-compactness.
On the contrary, it is a very useful fact that 
$ [ \mu, \lambda ]$-compactness can be studied in terms of 
$ [ \kappa  , \kappa  ]$-compactness, for $\mu \leq \kappa \leq \lambda $
(Remark \ref{notk}).

Notice that 
if $X$ is realized as a Tychonoff product $ \prod _{j \in J} X_j$, then
$\mathcal O ^{\Box} $-$ [ \mu, \lambda ]$-compactness,
as introduced in Definition \ref{fcpnopbox},
is the same as 
$\mathcal F$-$ [ \mu, \lambda ]$-compactness 
of $ \prod _{j \in J} X_j$,
when we take $\mathcal F$ to be the family of all open
sets in $ \bigbox _{j \in J} X_j$, that is, the open sets in the box topology.
\end{definition}

 If $\mathcal F$ is a family of subsets of some topological space,
we denote by  $ \bigvee \mathcal F$ (resp., $ \bigvee _{\leq \kappa }  \mathcal F$),
the family of all subsets of $X$ which can be obtained as the union of
the members of some subfamily of $\mathcal F$ (resp., some subfamily 
of cardinality $\leq \kappa $).

 \begin{theorem} \label{equivfdcpn} %5.2
 Suppose that $X$ is a topological space, $\mathcal F$ 
is a family of subsets of $X$, and 
  $\lambda$ and $\mu$ are infinite cardinals.
Then the following are equivalent.
\begin{enumerate} 
\item 
$X$ is
$\mathcal F$-$ [ \mu, \lambda ]$-compact.

\item 
 For every  sequence  $( P _ \alpha ) _{ \alpha \in \lambda } $
of subsets of $X$, if,
for every $Z \subseteq \lambda $ with $ |Z|< \mu$,
there exists some $F_Z \in\mathcal F$    such that   
$ \bigcap _{ \alpha \in Z}  P_ \alpha \supseteq F_Z$,
then  
$ \bigcap _{ \alpha \in \lambda }  \overline{P}_ \alpha \not= \emptyset $.

\item 
 For every  sequence  $( Q _ \alpha ) _{ \alpha \in \lambda } $
of sets in $\bigvee \mathcal F $
(equivalently, in $\bigvee \mathcal F _{\leq \kappa }$, 
where $\kappa= \lambda ^{<\mu}$), if,
for every $Z \subseteq \lambda $ with $ |Z|< \mu$,
there exists some $F_Z \in\mathcal F$    such that   
$ \bigcap _{ \alpha \in Z}  Q_ \alpha \supseteq F_Z$,
then  
$ \bigcap _{ \alpha \in \lambda }  \overline{Q}_ \alpha \not= \emptyset $.
The value of $\kappa$ can be improved to $\cf S_\mu ( \lambda )$. 
\item 
For every  sequence 
$\{ F_Z \mid Z \in S_ \mu ( \lambda ) \}$
of members of $\mathcal F$,
it happens that
    $ \bigcap _{ \alpha \in \lambda } 
 \overline{  \bigcup  \{ F_Z \mid  Z \in S_ \mu ( \lambda ),  \alpha \in Z \}}
  \not= \emptyset $.

\item 
For every  sequence 
$\{ F_Z \mid Z \in S_ \mu ( \lambda ) \}$
of members of $\mathcal F$,
the following holds.
If, for every finite subset $W$  of $\lambda$,
we put  
$Q_W= 
  \bigcup \{ F_Z \mid Z \in S_ \mu ( \lambda ) \text{ and } Z \supseteq
W\}$,
then 
$\bigcap \{ \overline{Q_W} \mid W \text{ is a finite }
 \brfr 
\text{subset of } \lambda  \}\not= \emptyset $.

\item
For every $\lambda$-indexed open cover
 $( Q _ \alpha ) _{ \alpha \in \lambda } $
 of $X$, there exists $Z \subseteq \lambda $, with $ |Z|< \mu $,
such that
$F \cap \bigcup _{ \alpha \in Z}  Q_ \alpha \not= \emptyset  $,
for every $F \in \mathcal F$.

\item
For every  sequence 
$\{ F_Z \mid Z \in S_ \mu ( \lambda ) \}$
of elements of $\mathcal F$,
there exists an ultrafilter $D$ over 
$S _\mu ( \lambda )$ which covers $\lambda$ and 
such that 
$\{ F_Z \mid Z \in S_ \mu ( \lambda ) \}$
has some $D$-limit point  in $X$.
\end{enumerate}
 \end{theorem}

 \begin{proof} 
Same as the proofs of  Proposition  \ref{equivo},
of the last remark in Definition \ref{fcpnop} 
 and of
 Theorem \ref{ufo}. See also 
Remark \ref{cofslm}.
\end{proof}  

Proposition \ref{equivo} and Theorem \ref{ufo}  can be obtained as the particular case of 
Theorem \ref{equivfdcpn}, when $\mathcal F= \mathcal O$ is the family
of the nonempty open sets of $X$.

Proposition \ref{equivobox} and Theorem \ref{ufobox}  can be obtained as the particular case of 
Theorem \ref{equivfdcpn}, when  $X $ is the topological space 
$ \prod _{j \in J} X_j$ (with the Tychonoff topology), and
 $\mathcal F$ is the family
of the nonempty open sets of 
$ \bigbox _{j \in J} X_j$ (with the box topology).

Thus, Theorem \ref{equivfdcpn} provides a generalization of 
all the above results.

As we mentioned in Remark \ref{caic}, in the particular case
when $\mathcal F$ is the family $\mathcal S$ of all singletons, the 
implication (1) $\Rightarrow $   (7) in Theorem \ref{equivfdcpn}
is proved in \cite{Cprepr,C}.  
Again when $\mathcal F =\mathcal S$, the equivalence of 
(1) and (2) in Theorem \ref{equivfdcpn} has been proved in \cite{G},
with different notation. See also 
\cite[Lemma 5(b)]{Vfund}.

\begin{theorem} \label{equivodcpnf} %5.4
 Suppose that $X$ is a topological space, $\mathcal F$ 
is a family of subsets of $X$, and 
$D$ is an ultrafilter over some set $I$.
Then the following are equivalent.
\begin{enumerate} 
\item 
$X$ is $\mathcal F$-$D$-compact.

\item
For every sequence 
$\{ F_i \mid i \in I \}$
of members of $\mathcal F$, if, for 
$Z \in D$, we put
$C_Z= \overline{ \bigcup _{i \in Z} F_i} $, then
we have that 
$\bigcap _{Z \in D} C_Z \not= \emptyset  $.

\item 
Whenever $(C_Z) _{Z \in D} $ is a  sequence  of closed sets of $X$
with the property that, for every $i \in I$, there exists some $F \in \mathcal F$  
such that $\bigcap _{i\in Z} C_Z \supseteq F$, then  
$\bigcap _{Z \in D} C_Z \not= \emptyset  $.

\item 
For every open cover $(O_Z) _{Z \in D} $ of $X$,
there is some $i \in I$ such that  
$F \cap \bigcup _{i\in Z} O_Z   \not= \emptyset $,
for every $F \in \mathcal F$.
\end{enumerate} 
\end{theorem} 

\begin{proof}
Similar to the proof of Theorem \ref{equivodcpn}.
 \end{proof}  

Theorem \ref{equivodcpn} could be obtained as the particular case 
$\mathcal F=\mathcal O$  of Theorem \ref{equivodcpnf}.
% ??? noto per $\mathcal F= Singlet???$ 

The particular case of Theorem \ref{equivodcpnf} when 
$\mathcal F$ is the set of all singletons of $X$ might be new,
so we state it explicitly.

\begin{corollary} \label{equivodcpnfs} 
 Suppose that $X$ is a topological space, 
and $D$ is an ultrafilter over some set $I$.
Then the following are equivalent.
\begin{enumerate} 
\item 
$X$ is $D$-compact.

\item
For every sequence 
$\{ x_i \mid i \in I \}$
of elements of $X$, if, for 
$Z \in D$, we put
$C_Z= \overline{ \{ x_i \mid i \in Z\}} $, then
we have that 
$\bigcap _{Z \in D} C_Z \not= \emptyset  $.

\item 
Whenever $(C_Z) _{Z \in D} $ is a  sequence  of closed sets of $X$
with the property that, for every $i \in I$,
$\bigcap _{i\in Z} C_Z \not= \emptyset  $, then  
$\bigcap _{Z \in D} C_Z \not= \emptyset  $.

\item 
For every open cover $(O_Z) _{Z \in D} $ of $X$,
there is some $i \in I$ such that  
$(O_Z) _{i \in Z } $
is a cover of $X$.
\end{enumerate} 
\end{corollary}   

\begin{theorem} \label{fprodolm} % 5.5
 Suppose that $\lambda$ and $\mu$ are infinite cardinals, $T$ is a family of topological spaces, 
and, for every $X \in T$, $\mathcal F_X$ is a family of subsets of $X$.

To every product $ \prod _{j \in J} X_j $, where each $X_j$ belongs to $T$,
associate the family $\mathcal F=
\{  \prod _{j \in J} F_j \mid F _{X_j}  \in \mathcal F_j \text{, for every  } j \in J \}$. 

Then 
the following are equivalent. 
\begin{enumerate} 
\item
There exists some ultrafilter 
 $D$ over  
$ S_ \mu ( \lambda )$ which covers $\lambda$, 
and such that, for every  $X \in T$, we have that $X$  is $\mathcal F_X$-$D$-compact.
\item 
There exists some $( \mu, \lambda  )$-regular ultrafilter $D$
(over any set)
such that, for every  $X \in T$, we have that $X$  is $\mathcal F_X$-$D$-compact.
\item 
There exists some $( \mu, \lambda  )$-regular ultrafilter $D$
such that, 
for every set $J$, 
every product $ \prod _{j \in J} X_j $ of members of $T$ (allowing repetitions)
is  $\mathcal F$-$D$-compact.
\item
For every set $J$, 
every product $ \prod _{j \in J} X_j $ of members of $T$ (allowing repetitions),
 is 
$\mathcal F $-$ [ \mu, \lambda ]$-compact.
\item
Let $\delta= \min \{ 2 ^{2^ \kappa  }, 
\sup \{ |T|, \sup _{X \in T} | \mathcal F_X| ^ \kappa  \}$,
where $\kappa= \lambda ^{<\mu}$
(indeed, this can be improved to $\kappa= \cf S_ \mu ( \lambda )$).  
For every set $J$ with $|J| \leq \delta $,
every product $ \prod _{j \in J} X_j $ of 
 members of $T$ (allowing repetitions)
 is 
$\mathcal F $-$ [ \mu, \lambda ]$-compact.
\end{enumerate}   
\end{theorem}

 \begin{proof} 
Same as the proofs  of  
Corollary \ref{regimpcpn} and of  
Theorem \ref{fprodo},
 using \cite[Fact 6.1 and Proposition 5.1 (b) with $\nu= |J^+|$]{LiF}
and Theorem \ref{equivfdcpn} (7) $\Leftrightarrow $  (1).
For (5), see also Remark \ref{cofslm}.
\end{proof}  

Theorem \ref{fprodolm} is more general than 
Theorems \ref{fprodo} and \ref{fprodot}.
In the particular case
when $\mathcal F$ is the family $\mathcal S$ of all singletons,
Theorem \ref{fprodolm} is essentially
\cite[Theorem 3.4]{C} (in some cases, our evaluation of $\delta$ 
 might be slightly sharper).
Corollaries \ref{regimpcpn} and \ref{regimpcpnbox} 
are  immediate consequences of Theorem \ref{fprodolm} (2) $\Rightarrow $  (4), by taking,
for every $j \in J$, 
$\mathcal F_j$  to be the family of all nonempty open sets
of $X_j$.

The following easy proposition,
generalizing Lemma \ref{lm}, 
 describes the behavior of 
$\mathcal F$-$D$-\brfrt compactness 
with respect to quotients.

\begin{proposition} \label{lemf} 
Suppose that  $X$ and $Y$ are topological spaces, and
 $f:X \to Y$ is a continuous function.
Suppose that $\mathcal F$ is a family of subsets of $X$,
and suppose that $\mathcal G$ is a family of subsets of $Y$,
such that for every $G \in \mathcal G$ there is
$F \in\mathcal F$ such that $F \subseteq f ^{-1}(G) $. 

Then the following hold.
\begin{enumerate}
\item
 If $X$ is   
$\mathcal F$-$ [ \mu, \lambda ]$-compact then 
$Y$ is $\mathcal G$-$ [ \mu, \lambda ]$-compact. 

\item
 If $X$ is   
$\mathcal F$-$ D$-compact then 
$Y$ is $\mathcal G$-$D$-compact. 
 \end{enumerate}  
\end{proposition} 

We end with a trivial but useful property of 
$\mathcal F$-$ [ \mu, \lambda ]$-compactness.

\begin{proposition} \label{ult} 
Every
$\mathcal F$-$ [ \cf \lambda , \cf \lambda ]$-compact 
topological space is
$\mathcal F$-$ [ \lambda , \lambda ]$-compact. 

In particular, 
every
$\mathcal O$-$ [ \cf \lambda , \cf \lambda ]$-compact 
topological space is
$\mathcal O$-$ [ \lambda , \lambda ]$-compact. 
\end{proposition}

\bibliographystyle{plain}

\begin{thebibliography}{KaMa}

\bibitem[Ar]{Ar} A. V. Arhangel{$'$}skii,  \emph{Strongly
   $\tau$-pseudocompact spaces}, Topology Appl. \textbf{89} (1998),
   285--298.

%\bibitem[AlUr]{AU} P. Alexandroff, P. Urysohn, {\em M\'emorie sur les \'espaces topologiques compacts}, Ver. Akad. Wetensch. Amsterdam \textbf{14} (1929), 1-96. 

  

\bibitem[Ca1]{Cprepr} X. Caicedo, {\em On productive $[\kappa,\lambda]$-compactness, or  the Abstract Compactness Theorem revisited},  manuscript (1995). 

 


\bibitem[Ca2]{C} X. Caicedo, {\em The Abstract Compactness Theorem Revisited}, in {\em Logic and Foundations of Mathematics (A. Cantini et al. editors),} Kluwer Academic Publishers  (1999), 131--141. 


 


%\bibitem[CoNe]{CN}  W. Comfort, S. Negrepontis, {\em The Theory of Ultrafilters}, Berlin (1974). 
 

\bibitem[CoNe]{CNcc}  W. Comfort, S. Negrepontis, {\em     Chain conditions in topology}, Cambridge Tracts in Mathematics \textbf{79},
   Cambridge University Press, Cambridge-New York (1982). 

 

%\bibitem[Do] {Do} H.-D. Donder,  \emph{Regularity of ultrafilters and the core model}, Israel J. Math. {\bf 63} (1988), 289--322.

 

\bibitem
 [Fr]{Fr} Z. Frol{\'{\i}}k,  
%%%{\cyr Obobweniya kompaktnosti i svo...   }
\emph{
Generalisations of compact and Lindel\"of spaces} (Russian, with expanded English summary),
  Czechoslovak Math. J.
\textbf{9} (1959), 172--217.

 
\bibitem[G\'a]{G} 
I. S. G{\'a}l, 
 \emph{On the theory of $(m,\,n)$-compact topological spaces},
 Pacific J. Math.
 \textbf{8}
  (1958),
721--734.

%\bibitem [Ga1]{GF1} S. Garcia-Ferreira, \emph{Some remarks on initial $\alpha$-compactness, $<\alpha$-boundedness   and $p$-compactness},   Topology Proc. {\bf 15}  (1990), 11--28.

 


\bibitem
 [Ga]{GF} S. Garcia-Ferreira, \emph{On two generalizations of 
pseudocompactness},
Topology Proc. {\bf 24} (Proceedings of the 14$^\text{th}$
   Summer Conference on General Topology and its Applications
  Held at Long Island University, Brookville, NY, August 4--8,
1999) (2001), 149--172. 

 



\bibitem
[GiSa]{GS} J. Ginsburg and V. Saks,
\emph{Some applications of ultrafilters in topology},
Pacific J. Math. {\bf  57} (1975), 403--418.

 


\bibitem 
[Gl] {Gl} 
I. Glicksberg, \emph{Stone-\v Cech compactifications of products}, Trans. Amer. Math. Soc
{\bf 90} (1959),  369--382.

 


%\bibitem[HNV]{EGT} K. P. Hart, J. Nagata, J. E. Vaughan (editors), {\em Encyclopedia of General Topology}, Amsterdam (2003). 

\bibitem[Ka]{Ka}  A.  Kanamori, 
\emph{Finest partitions for ultrafilters},  
   J. Symbolic Logic \textbf{51} (1986),   327--332.
 
\bibitem[Ke]{Ke}  J. F. Kennison, 
   \emph{$m$-pseudocompactness},
   Trans. Amer. Math. Soc. \textbf{104} (1962), 436--442.
 

 

%\bibitem[KuVa] {KV} K. Kunen and J. E. Vaughan (editors), {\em  Handbook of Set Theoretical Topology}, Amsterdam (1984).

\bibitem[Li1] {topproc} P. Lipparini,  {\em Productive $[\lambda,\mu ]$-compactness and regular ultrafilters}, Topology Proceedings \textbf{21} (1996), 161--171.
 
%\bibitem[L5] {abst} P. Lipparini,  {\em Regular ultrafilters and $[ \lambda , \lambda ]$-compact products of topological spaces} (abstract), Bull. Symbolic Logic \textbf{5} (1999), 121.


\bibitem[Li2]{topappl} P. Lipparini,  {\em Compact factors in finally compact products of topological spaces}, Topology and its Applications \textbf{153} (2006), 1365--1382.

 


%\bibitem[Li3]{nuotop} P. Lipparini,  {\em Combinatorial and model-theoretical principles related to regularity of ultrafilters and compactness of topological spaces. I}, arXiv:0803.3498; {\em II.}:0804.1445; {\em III.}:0804.3737; {\em IV.}:0805.1548  (2008); {\em V.}:0903.4691; {\em VI.}:0904.3104  (2009).

 


\bibitem[Li3]{mru} P. Lipparini,  \emph{More   on regular and decomposable ultrafilters in ZFC}, accepted by Mathematical Logic Quarterly, preprint available on arXiv:0810.5587 (2008).



%\bibitem[Li5]{dec} P. Lipparini,  \emph{Transfer of ultrafilter decomposability}, in preparation.



\bibitem[Li4]{LiF} P. Lipparini,  
\emph{Some compactness properties related to pseudocompactness
and ultrafilter convergence}, submitted, preprint available at 
arXiv:0907.0602v3


 
\bibitem[Re]{Re}
T. Retta, \emph{Some cardinal
   generalizations of pseudocompactness}, Czechoslovak Math. J. \textbf{43}
   (1993), 385--390. 
 

\bibitem
[Sa]
{Sa} V. Saks,   \emph{Ultrafilter invariants in
   topological spaces}, Trans. Amer. Math. Soc. \textbf{241} (1978), 79--97.

 

\bibitem
[SaSt]
{SS2} V. Saks,  R. M. Stephenson Jr, \emph{Products of $\m {M}$-compact spaces}, Proc. Amer. Math. Soc. {\bf 28}  (1971),  279--288.

%\bibitem[Ste]{steph} R. M. Stephenson, {\em Initially $\kappa $-compact
% and related spaces}, Chapter 13 in \cite{KV}, 603--632. 


\bibitem[ScSt]{SS}  C. T.  Scarborough, A. H. Stone, \emph{Products of     
   nearly compact spaces}, Trans. Amer. Math. Soc. {\bf 124}  (1966),  131--147.  

 \bibitem[Sh]{Sh} S. Shelah, \emph{Cardinal arithmetic},
Oxford
   Logic Guides, Vol. \textbf{29}, Oxford Science Publications. The Clarendon Press,
   Oxford University Press, New York (1994). 


\bibitem[St]{St} R. M. Stephenson Jr, \emph{Pseudocompact spaces}, ch. d-07 in 
\emph{Encyclopedia of general topology}, Edited by
   K. P. Hart, J. Nagata and J. E. Vaughan. Elsevier
   Science Publishers, B.V., Amsterdam (2004).

 

\bibitem[StVa]{SV} R. M. Stephenson Jr, J. E. Vaughan, \emph{Products of initially $\m m$-compact spaces}, Trans. Amer. Math. Soc. {\bf 196}  (1974), 177--189.  
 

 

%\bibitem[Va1]{VLNM} J. E. Vaughan, {\em Some recent results in the theory of  [a,b]-compactness},in {\em TOPO 72---General Topology and its Applications} (Proc. Second Pittsburg Internat. Conf., Carnegie-Mellon Univ. and Univ. Pittsburg, 1972), Lecture Notes in Mathematics  \textbf{378} (1974), 534--550. 


 \bibitem[Va]{Vfund} J. E. Vaughan, {\em Some properties
 related to [a,b]-compactness}, Fund. Math. \textbf{87} (1975), 251--260. 


% Vaughan, Heredity of tau-pseudocompactness  Scientiae Mathematicae Japonicae scaricato 

\end{thebibliography}

\end{document}